\documentclass[12pt]{amsart}
\usepackage{amssymb}
\usepackage{amsfonts}
\usepackage{amscd}
\usepackage[all]{xy}
\swapnumbers

\def\ac{{\rm ac}}
\def\pr{{\rm pr}}

\def\GM{{\mathbb G}_m}

\def\AA{{\mathbb A}}

\def\CC{{\mathbf C}}

\def\QQ{{\mathbf Q}}

\def\cL{{\mathcal L}}
\def\cM{{\mathcal M}}

\def\cS{{\mathcal S}}

\def\cX{{\mathcal X}}

\mathchardef\alphag="7C0B \mathchardef\betag="7C0C
\mathchardef\gammag="7C0D \mathchardef\deltag="7C0E
\mathchardef\varepsilong="7C22 \mathchardef\varphig="7C27
\mathchardef\psig="7C20 \mathchardef\zetag="7C10
\mathchardef\epsilong="7C0F \mathchardef\rhog="7C1A
\mathchardef\taug="7C1C \mathchardef\upsilong="7C1D
\mathchardef\iotag="7C13 \mathchardef\thetag="7C12
\mathchardef\pig="7C19 \mathchardef\sigmag="7C1B
\mathchardef\etag="7C11 \mathchardef\omegag="7C21
\mathchardef\kappag="7C14 \mathchardef\lambdag="7C15
\mathchardef\mug="7C16 \mathchardef\xig="7C18
\mathchardef\chig="7C1F \mathchardef\nug="7C17
\mathchardef\varthetag="7C23 \mathchardef\varpig="7C24
\mathchardef\varrhog="7C25 \mathchardef\varsigmag="7C26
\mathchardef\Omegag="7C0A \mathchardef\Thetag="7C02
\mathchardef\Sigmag="7C06 \mathchardef\Deltag="7C01
\mathchardef\Phig="7C08 \mathchardef\Gammag="7C00
\mathchardef\Psig="7C09 \mathchardef\Lambdag="7C03
\mathchardef\Xig="7C04 \mathchardef\Pig="7C05
\mathchardef\Upsilong="7C07

\newtheorem{theorem}[subsection]{Theorem}
\newtheorem{lem}[subsection]{Lemma}

\theoremstyle{definition}
\newtheorem{definition}[subsection]{Definition}

\newtheorem{def-prop}[subsubsection]{Definition-Proposition}

\theoremstyle{remark}

\theoremstyle{plain}

\numberwithin{equation}{subsection}

\def\boxit#1#2{\setbox1=\hbox{\kern#1{#2}\kern#1}%
\dimen1=\ht1 \advance\dimen1 by #1 \dimen2=\dp1 \advance\dimen2 by
#1
\setbox1=\hbox{\vrule height\dimen1 depth\dimen2\box1\vrule}%
\setbox1=\vbox{\hrule\box1\hrule}%
\advance\dimen1 by .4pt \ht1=\dimen1 \advance\dimen2 by .4pt
\dp1=\dimen2 \box1\relax}

\renewcommand{\theequation}{\thesubsection.\arabic{equation}}

\def\AA{{\mathbb A}}

\def\CC{{\mathbb C}}

\def\LL{{\mathbb L}}

\def\QQ{{\mathbb Q}}

\def\cL{{\mathcal L}}
\def\cM{{\mathcal M}}

\def\cS{{\mathcal S}}

\def\cX{{\mathcal X}}

\mathchardef\alphag="7C0B \mathchardef\betag="7C0C
\mathchardef\gammag="7C0D \mathchardef\deltag="7C0E
\mathchardef\varepsilong="7C22 \mathchardef\varphig="7C27
\mathchardef\psig="7C20 \mathchardef\zetag="7C10
\mathchardef\epsilong="7C0F \mathchardef\rhog="7C1A
\mathchardef\taug="7C1C \mathchardef\upsilong="7C1D
\mathchardef\iotag="7C13 \mathchardef\thetag="7C12
\mathchardef\pig="7C19 \mathchardef\sigmag="7C1B
\mathchardef\etag="7C11 \mathchardef\omegag="7C21
\mathchardef\kappag="7C14 \mathchardef\lambdag="7C15
\mathchardef\mug="7C16 \mathchardef\xig="7C18
\mathchardef\chig="7C1F \mathchardef\nug="7C17
\mathchardef\varthetag="7C23 \mathchardef\varpig="7C24
\mathchardef\varrhog="7C25 \mathchardef\varsigmag="7C26
\mathchardef\Omegag="7C0A \mathchardef\Thetag="7C02
\mathchardef\Sigmag="7C06 \mathchardef\Deltag="7C01
\mathchardef\Phig="7C08 \mathchardef\Gammag="7C00
\mathchardef\Psig="7C09 \mathchardef\Lambdag="7C03
\mathchardef\Xig="7C04 \mathchardef\Pig="7C05
\mathchardef\Upsilong="7C07

\def\ord{{\rm ord}}

\begin{document}\title[Composition with a two variable function]
{Composition with a two variable function}

\author{Gil Guibert}

\address{39 quai du Halage,
94000 Cr\'eteil, France}

\email{guibert9@wanadoo.fr}

\author{Fran\c cois Loeser}

\address{{\'E}cole Normale Sup{\'e}rieure,
D{\'e}partement de math{\'e}matiques et applications, 45 rue
d'Ulm, 75230 Paris Cedex 05, France (UMR 8553 du CNRS)}
\email{Francois.Loeser@ens.fr}
\urladdr{http://www.dma.ens.fr/~loeser/}

\author{Michel Merle}

\address{Laboratoire J.-A. Dieudonn\'e,
Universit\'e de Nice - Sophia Antipolis, Parc Valrose, 06108 Nice
Cedex 02, France (UMR 6621 du CNRS)} \email{Michel.Merle@unice.fr}
\urladdr{http://www-math.unice.fr/membres/merle.html}

\maketitle

\begin{center}
{\Large\bf}

\vspace{3mm}
\end{center}
\vspace{3mm}
\section{Introduction}In  \cite{nemethi1} and \cite{nemethi2},
A. N\'emethi studied the Milnor fiber and monodromy zeta function
of composed functions of the form $f (g_1, g_2)$ with $f$ a two
variable polynomial and $g_1$ and $g_2$ polynomials with distinct
sets of variables. The present paper addresses the question of
proving similar results for the motivic Milnor fiber introduced by
Denef and Loeser, cf.
\cite{motivic},\cite{barc},\cite{looi},\cite{lef}. In fact,
N\'emethi later considered in \cite{nemethi3} the more general
situation of a composition $f \circ  \mathbf{g} \colon (X,x)\to
(\CC^2,0)\to(\CC,0)$, where $\mathbf{g}$ has a reasonable
discriminant.
 Still later, N\'emethi and
Steenbrink \cite{ns} proved similar results at the level of the
Hodge spectrum \cite{St1},\cite{St2},\cite{Va}, using the theory
of mixed Hodge modules. In particular, they were  able to compute,
under mild assumptions, the Hodge spectrum  of composed functions
of the form
 $f (g_1, g_2)$ without assuming
the variables in $g_1$ and $g_2$ are distinct. Their result
involves the discriminant of the morphism ${\bf g}Ê= (g_1, g_2)$.
In a previous paper \cite{ivc}, we computed the motivic Milnor
fiber for functions of the form $g_1 + g_2^{\ell}$ when $\ell$ is
large without assuming the variables in $g_1$ and $g_2$ are
distinct. The corresponding result for the Hodge spectrum goes
back to M. Saito \cite{saito} and is a special case of the results
of N\'emethi and Steenbrink \cite{ns}. So, it seems very natural
to search for a full motivic analogue of the results of \cite{ns}.
At the present time, we are unable to realize this program and we
have to limit ourself, as we  already mentioned, to the case when
$g_1$ and $g_2$ have no variable in common. Already extending our
result to the case when one only assumes the discriminant of the
morphism ${\bf g}$ is contained in the coordinate axes seems to
require new ideas.

\bigskip

In this paper we consider a polynomial $f$ in $k[x,y]$ and we
assume that $f(0,y)$ is non zero of degree $m$. We denote by $i_p$
the closed embedding into $\AA^2_k$ of a point $p$ in
$F_0=f^{-1}(0)\cap x^{-1}(0)$. We consider the motivic Milnor
fiber $\cS_f$ of the function $f: \AA^2_k\longrightarrow \AA^1_k$
whose restriction $i_p^* \cS_{f}$ above $p$ is an element of the
Grothendieck ring $\cM_{\GM}^{\GM}$. We then reformulate Guibert's
computation of the motivic Milnor fiber of germs of plane curve
singularities \cite{guibert} using generalized convolution
operators of \cite{glm}. More precisely, we express it in terms of
the tree $\tau(f,p)$ associated  to $f$, depending on the given
coordinate system $(x,y)$ on the affine plane $\AA^2_k$. Let us
recall the tree $\tau(f,p)$ is obtained by considering the Puiseux
expansions of the roots of $f$ at $p$, cf.
\cite{kl},\cite{eggers}.  To any so-called rupture vertex $v$ of
this graph, we attach a weighted homogeneous polynomial $Q_{v,f}$
in $k[c,d]$. We have defined in \cite{glm} a generalized
convolution by such a polynomial. It is a morphism from $\cM_{\GM
\times \GM}^{\GM}$ to $\cM_{\GM}^{\GM}$, but can be extended to a
morphism from $\cM_{\AA_k^1\times \GM}^{\GM}$ to
$\cM_{\GM}^{\GM}$. We denote by $\varpi_j$ the morphism
$x\longmapsto x^j$ from $\GM$ to $\GM$ and by $m_p$ the order of
$p$ as a root of $f(0,y)$. One can then reformulate Guibert's
theorem as
\begin{equation}\label{guib}\tag{$\ast$}
i_p^* \cS_{f}=[\varpi_{m_p} :\GM\longrightarrow \GM]-\sum_v
\Psi_{Q_{v,f}}( [\mathrm{Id}:\AA_k^1\times \GM \longrightarrow
\AA_k^1\times \GM ])
\end{equation}
where the sum runs over the set of rupture vertices of
$\tau(f,p)$.

For $1\leq j\leq 2$, let $g_j:X_j \longrightarrow \AA_k^1$ be a
function on a smooth $k$-variety $X_j$. By composition with the
projection, $g_j$ becomes a function on the product $X=X_1\times
X_2$ and we  write  $\mathbf{g}$ for  the map $g_1\times g_2 : X
\rightarrow \AA^2_k$. The main result of this paper, Theorem
\ref{main}, gives a formula for
$i^* \cS_{f\circ \mathbf{g}}$,  
where $i$ denote the inclusion of $g_1^{-1} (0) \cap
g_2^{-1}Ê(0)$,
 similar to (\ref{guib}),
with $[\mathrm{Id}:\AA_k^1\times \GM \longrightarrow \AA_k^1\times
\GM ]$ replaced by a virtual object $A_v$. The virtual object
$A_v$ is
 defined inductively in terms of the tree
 associated to $f$ at the origin of $\AA^2_k$
 and
of $A_{v_0}$, where $v_0$ is the first (extended) rupture vertex
of $\tau(f,p)$, and $A_{v_0}$ depends only on $\mathbf{g}$.

\section{Preliminaries and combinatorial set up}\label{section2}


\subsection{}We fix an algebraically closed field $k$ of characteristic 0.
For a variety $X$ over $k$, we denote by $\cL (X)$ and $\cL_n (X)$
the spaces of arcs, resp. arcs mod $t^{n+1}$ as defined in
\cite{arcs}. As in \cite{ivc}, we denote by $\cM_X$ the
localisation of the Grothendieck ring of varieties over $X$ with
respect to the class of the relative line. We shall also use the
$\GM$-equivariant variant $\cM_{X \times \GM^p}^{\GM}$ defined in
\cite{glm}, which is generated by classes of objects $Y
\rightarrow X \times \GM^p$ endowed with a monomial $\GM$-action.

Also, if $p$ is a closed point of $X$ we denote by $i_p$ the
inclusion $i_p : p \rightarrow X$ and by $i_p^*$ the corresponding
pullback morphism at the level of rings $\cM$.

\subsection{}Let us start by recalling some basic
constructions introduced by Denef and Loeser in \cite{motivic},
\cite {lef} and \cite{barc}.

Let $X$ be a smooth variety over $k$ of pure dimension $d$ and $g
: X \rightarrow \AA^1_k$. We set $X_0 (g)$ for the zero locus of
$g$, and consider, for $n \geq 1$, the variety
\begin{equation}
\cX_n (g) := \Bigl \{\varphi \in \cL_n (X) \Bigm  \vert \ord_t g
(\varphi) = n \Bigr\}.
\end{equation}
Note that $\cX_n (g)$ is invariant by the $\GM$-action on $\cL_n
(X)$. Furthermore $g$ induces a morphism $g_n  : \cX_n (g)
\rightarrow \GM$, assigning to a point $\varphi$ in $\cL_n (X)$
the coefficient $\ac (g (\varphi))$ of $t^n$ in $g (\varphi)$,
which we shall also denote by $\ac (g) (\varphi)$. This morphism
is homogeneous of weight $n$ with respect to the $\GM$-action on
$\cX_n (g)$ since $g_n (a \cdot \varphi) = a^n g_n (\varphi)$, so
we can consider the class $[\cX_n (g)]$ of $\cX_n (g)$ in
$\cM_{X_0 (g) \times \GM}^{\GM}$.

We now consider the motivic zeta function
\begin{equation}
 Z_g(T): =\sum_{n\geq 1}[\cX_n(g)]\,\LL^{-nd}\,T^n
\end{equation}
in $\cM_{X_0 (g) \times \GM}^{\GM}[[T]]$. Note that $Z_g = 0$  if
$g = 0$ on $X$.

Denef and Loeser showed in \cite{motivic} and \cite{barc} (see
also \cite{lef}) that $Z_g(T)$ is a rational series by giving a
formula for $Z_g(T)$ in terms of a resolution of $f$. They also
showed that one can consider $ \lim_{T \mapsto \infty} Z_g (T)$ in
$\cM_{X_0 (g) \times \GM}^{\GM}$ and they define the motivic
Milnor fiber of $g$ as
\begin{equation}
\cS_g := - \lim_{T \mapsto \infty} Z_g (T).
\end{equation}

\subsection{}In this subsection we do not assume $X$ to be  smooth.
For technical reasons we shall use in the present paper the
following innocuous variant of $\cM_{X \times \GM^{p}}^{\GM}$:
replacing everywhere in the definition the first $\GM$-factor
endowed with the $\GM$-action by multiplicative translation
$\lambda \cdot x = \lambda x$ by $\AA^1_k$ with ``the same"
$\GM$-action one gets a ring $\cM_{X
 \times \AA^1_k \times \GM^{p- 1}}^{\GM}$ generated by classes of
objects $Y \rightarrow X \times \AA^1_k \times \GM^{p- 1} $
endowed with a monomial $\GM$-action.

If $Q$ is a quasihomogeneous polynomial in $p$ variables, we
defined in \cite{glm} a convolution operator
\[
\Psi_Q : \cM_{X \times \GM^{p}}^{\GM} \longrightarrow \cM_{X
\times \GM}^{\GM}.
\]
In this paper we shall use the slight variant, still denoted by
$\Psi_Q $, which is obtained with the same definition, replacing
$\GM^{p}$ by $\AA^1_k \times \GM^{p- 1}$,
\[
\Psi_Q : \cM_{X \times \AA^1_k \times \GM^{p- 1} }^{\GM}
\longrightarrow \cM_{X \times \GM}^{\GM}.
\]
In fact, such  constructions carry over for any toric variety with
torus $\GM^p$, not only for $ \AA^1_k \times \GM^{p- 1}$.

\subsection{}Fix a positive integer $N$ and consider the
ring of fractional power series $k[[x^\frac{1}{N}]]$. Given  a
positive rational number $r$ we denote by $I_{\geq r}$ the ideal
of power series of order at least $r$ in $k[[x^\frac{1}{N}]]$. We
call the quotient $k[[x^\frac{1}{N}]]/I_{\geq r}$ the ring of
$r$-truncated fractional power series.

To a $r$-truncated fractional power series $y$ one assigns a
labelled rooted real metric tree $\tau_r(y)$ in the following way.
The total space $\tau_r(y)$ is  the half-open interval $[0,r)$ and
its  vertices are the positive exponents with non zero
coefficients of the expansion of $y$ in powers of $x$ together
with the origin which is the root. We define the  \emph{height} of
a vertex to be its distance to the root. We label each vertex by
the coefficient of the corresponding term (this coefficient is
non-zero for all vertices except maybe for  the root). The
vertices are ordered by the height. Starting above a vertex there
is only one edge. This edge ends with the next vertex if there is
one and remains open above the last vertex. We label each edge by
0 and we say
 $\tau_r(y)$ is of height $r$. We denote by $|\tau_r(y)|$ the
underlying unlabelled tree. Notice that we can see the labels as
degree 1 polynomials of $k[X]$ (or cycles in $\AA_k^1$), that is,
$X$ for an edge and $X-a$ for a vertex labelled by $a$.

If now $y$ is a power series in $k[[x^\frac{1}{N}]]$, we denote by
$\tau_r(y)$ the height $r$ tree associated its  truncation of $y$
at order $r$. Thus, for $r<r'$, $\tau_r(y)$ is obtained from
$\tau_{r'}(y)$ by truncating up to height $r$. We denote by $\tau
(y)$ the inductive limit of the system $(\tau_r(y))_{r\in \QQ}$
and call it the  tree associated to the power series $y$.

\subsection{}We consider a two variable polynomial $f$ in
$k[x,y]$. We assume that $f(0,y)$ is non zero of degree $m$ and we
consider the $m$ Newton-Puiseux expansions $y_i$, 
$1 \leq i \leq m$, associated to $f$ at the points of
$f^{-1}(0)\cap x^{-1}(0)$. There exists an integer $N$ such that
these roots are elements of the ring of fractional power series
$k[[x^\frac{1}{N}]]$, namely, they are the roots of the polynomial
$f$ in $k[[x^\frac{1}{N}]]$.

Fix a positive rational number $r$. We denote by $\bigcup_{i=1}^m
\tau_r(y_i)$ the labelled rooted real metric tree which
 is obtained as follows.  In the disjoint union of
the trees $\coprod_{i=1}^m |\tau_r(y_i)|$ we identify two vertices
(resp. two edges) if they have same height and same label. If $v$
is a vertex (resp. an edge) shared by trees $\tau_r(y_i)$ for $i$
in $J$, then its label on the union is the $|J|$-th power of its
label on any of the $\tau_r(y_i)$, $i$ in $J$.


Since the  group of $N$-roots of unity acts on the $r$-truncated
expansions $(y_1,\ldots, y_m)$, it also acts on  $\vert
\bigcup_{i=1}^m \tau_r(y_i)\vert$. We denote by $|\tau_r(f)|$ the
separated quotient and by $\pi : \vert \bigcup_{i=1}^m
\tau_r(y_i)\vert \rightarrow |\tau_r(f)|$ the quotient morphism.
Note that the connected components of $|\tau_r(f)|$ are in natural
bijection with points of $f^{-1}(0)\cap x^{-1}(0)$. For any such
point $p$, we denote by $\vert \tau(f,p)\vert $ the corresponding
connected component which is naturally endowed with the structure
of a rooted real metric tree.
We attach labels to the vertices and edges of $|\tau_r(f)|$ in the
following way:
 \begin{itemize}
\item[$\bullet$] If $e$ is an edge of $|\tau_r(f)|$, the label
attached to $e$ is the label on any element of $\pi^{-1}(e)$. It
is a power of $X$ in $k[X]$ and we denote it by $P_{e,f}$. We will
call \emph{degree of the edge} $e$ the degree of $P_{e,f}$.
\item[$\bullet$]  If $v$ is a vertex
of $|\tau_r(f)|$, the label on $v$ is the product  of the labels
on $\pi^{-1}(v)$. We denote it by $P_{v,f}$. We will call
\emph{degree of the vertex} $v$ the degree of $P_{v,f}$. Notice
that the degree of a vertex $v$ is equal to the degree of the edge
$e$ which ends in $v$.
\end{itemize}


For $r<r'$, the graph $\tau_r(f)$ is the truncation of
$\tau_{r'}(f)$ at height $r$. The \emph{graph of contacts}
$\tau(f)$ defined by $f$ along $f^{-1}(0)\cap x^{-1}(0)$ is the
inductive limit of the graphs $\tau_r(f)$, $r\in \QQ$, cf.
\cite{kl}, \cite{eggers}. We say that a vertex $v$ of $\tau(f)$ is
a \emph{rupture vertex} if the set of zeroes of $P_{v,f}$ contains
at least two points in $\AA_k^1$. We define the augmented set of
rupture vertices of the tree $\tau(f,p)$ as the set of rupture
vertices of $\tau(f,p)$ together with the
vertex of minimal non zero height on $\tau(f,p)$.

We fix from now on a point $p$ which will be assumed for
simplicity to be the origin in $\AA_k^2$. For any arc $\varphi$ in
$\cL(\AA^2_k)$ such that $\varphi(0)=p$ and $x(\varphi)\neq 0$,
there exist power series $\omega$ in $k[[t]]$ and $\sum_j b_j
\omega^j$ in $k[[\omega]]$, and an integer $M$ such that
$\mathrm{gcd}(M, \{j \mid b_j\neq 0\})=1$ and
\begin{equation*}\begin{array}{rcl}
x(\varphi(t)) &=& \omega(t)^M\\ y(\varphi(t))&=&\sum_j b_j
\,\omega(t)^j.
\end{array}
\end{equation*}
Hence
 \[
y(\varphi(t))=\sum_j b_j (x(\varphi(t))^{\frac{j}{M}}
 \]
  is a fractional power series in
$x(\varphi(t))$. We consider the tree $\tau(y)$ with
 \[y(x):=\sum_j b_j
x^{\frac{j}{M}}\] in $k[[x^{\frac{1}{M}}]]$.

The power series  $\omega$ is defined up to an
 $M$-rooth of unity so
that the $b_j$'s are defined up to a factor $\zeta^j$ with $\zeta$
an $M$-root of unity. Two different choices lead to trees in the
same $\mu_N$-orbit. This orbit is is denoted by $\tau (\varphi)$.
Notice that $\tau(\varphi)$, as well as $\tau(f)$, depends on the
system of coordinates $(x,y)$.

\begin{definition}Consider $\varphi$ in $\cL(\AA^2_k)$ and
$f$ in $k[x,y]$ as before. The \emph{order of contact} of
$\varphi$ with $f$ is the maximum number $s$ in $\QQ\cup
\{\infty\}$ such that $\tau_s(\varphi)$ is included in $\tau_s(f)$
(it is infinite if and only if $f(\varphi)=0$). The \emph{contact}
of $\varphi$ with $f$ is the tree $\tau_r(\varphi)$ where $r$ is
the order of contact of $\varphi$ with $f$.
\end{definition}

\subsection{} From now on the polynomial $f$ is fixed in
$k[x,y]$ and we denote by $m$ the degree of $f(0,y)$. For a
positive rational number $r$, by a  contact $\tau$ of order $r$,
we mean a subtree of $\tau_r(f,p)$ which is isomorphic to $[0,
r)$. In particular $\tau$ is rooted at $p$ and its closure in
$\tau_r(f,p)$ contains a unique point  of height $r$, not
necessarily a vertex of $\tau_r(f,p)$, which completely determines
$\tau$. To such a contact $\tau$  we assign a polynomial
$P_{\tau,f}$ in the following way. The last and (semi)open edge of
$\tau$ is contained in a unique edge $e$ of $\tau_r(f,p)$.
\begin{itemize}
\item[$\bullet$] If $e$ ends at a vertex $v$ at height $r$
of $\tau_r(f,p)$ (in this case we  say that $\tau$ ends at the
vertex $v$), we will set $P_{\tau,f}=P_{\tau,v}$.
\item[$\bullet$] Otherwise,
(in that case we say that $\tau$ ends at the edge $e$) we set
$P_{\tau,f}=P_{\tau,e}$.
\end{itemize}
By definition of contact, there is an integer $M$ and a polynomial
$y_\tau$ in $k[\omega]$, of degree strictly smaller than $rM$,
both depending only on $\tau$, such that for any arc $\varphi$ of
contact $\tau$ with $f$, there exists a series $\omega$ in
$k[[t]]$, $\ord_t(\omega)=\ell$,  such that
\begin{equation*}\label{contact}
\begin{array}{rcl}
x(\varphi(t))&=&\omega(t)^M\\
y(\varphi(t))&=&y_\tau(\omega(t))\quad [\mathrm{mod}(t^{\lceil
rM\ell \rceil})].
\end{array}
\end{equation*}
For an arc $\varphi$ of contact $\tau$ with $f$, the quotient
$\ord_t (f(\varphi))/\ell$ is an integer and depends only on
$\tau$. We denote it by $\nu(\tau)$. One always has the inequality
$\nu(\tau)\geq Mr$.

The tree $\tau(f,p)$ is built from the Puiseux expansions of the
$m$ roots of $f(x,y)$ in the ring of fractional power series
$\bigcup_N k[[x^{1/N}]]$. Conversely, to any finite subtree
$\varsigma$ of $\tau(f,p)$, we can associate a polynomial
$f_\varsigma$ in $k[x,y]$ which is the minimal polynomial of the
$m$ Puiseux expansions restricted to $\varsigma$. Considering the
tree associated to the polynomial $f_\varsigma$, we get a tree
$\tau(f_\varsigma, p)$ which is an infinite tree with a finite
number of vertices. The intersection of $\tau(f_\varsigma, p)$
with $\tau(f,p)$ contains $\varsigma$. As an example, we can
consider the tree $\tau_r$ obtained from $\tau(f,p)$ by truncation
at height $r$. We will denote by $\overline{\tau_r}$ the tree
$\tau(f_{\tau_r}, p)$. 

\section{Guibert's theorem revisited}

\subsection{}We consider the following set of arcs:


\[
\cX_{\tau,\ell} := \Bigl \{\varphi \in \cL_{\bar{\nu}(\tau)\ell}
(\AA^2_k) \Bigm \vert \varphi \mbox{ has contact $\tau$ with $f$,
}\;  \ord_tx(\varphi)=M\ell\Bigr\}.
\]

where $\bar{\nu}(\tau )$ is the maximum of the integers $\nu (\tau
)$ and $M$.

\medskip


We denote by $Q_{\tau,f}$ the function $\omega^{\nu(\tau)}
P_{\tau,f}(\omega^{-Mr}c)$. One should note  that $Q_{\tau,f}$ is
a polynomial in $k[c,\omega]$, even if $Mr$ may not be an integer.
\begin{lem}\label{isoxrl}Consider a contact $\tau$ and an integer $\ell$ and
denote by $N(\tau,\ell)$ the integer $2\bar{\nu}(\tau)\ell -M\ell
- \lfloor Mr\ell\rfloor $. For any arc $\varphi$ in
$\cX_{\tau,\ell}$, there exist two series $\omega$ and
$\varepsilon$ in $k[t]/t^{\bar{\nu}(\tau)\ell+1}$ such that
\begin{enumerate}
    \item[(1)] $\ord_t(\omega)=\ell$,  $x(\varphi)=\omega^M$
    \item[(2)] $\ord_t(\varepsilon)\geq Mr\ell$ (resp. $=Mr\ell$ if $\tau$
    ends in an edge),
    $y(\varphi)=y_\tau(\omega) +\varepsilon$.
\end{enumerate}
The mapping
$(\varepsilon,\omega)\longmapsto(\omega^M,y_\tau(\omega)+\varepsilon)$
induces an isomorphism
\begin{equation*}
\Phi :
(\AA_k^1\times \GM)\setminus Q_{\tau,f})
^{-1}(0)\times\AA_k^{N(\tau,\ell)} \longrightarrow
\cX_{\tau,\ell}
\end{equation*}
given by
\begin{equation*}
(c, \omega_\ell,a)  \longmapsto (t^{\ell M}
(\omega_\ell +\sum_{k=1}^{\ell (\bar{\nu}(\tau )-M)} a_k t^k )^M
[t^{\ell \bar{\nu}(\tau )+1}],
y_\tau(\omega)+ct^{Mr\ell}+\sum_{\ell (\bar{\nu}(\tau )-M)<k\leq
\ell \bar{\nu}(\tau )} a_k t^k [t^{\ell \bar{\nu}(\tau )+1}] ).
\end{equation*}
Via the isomorphism $\Phi$, the angular coefficient $\ac
(f(\varphi))$ is given, up to a non-zero constant, by 
the following formula:
\begin{equation*}\label{}
\ac (f(\varphi))\sim 
\omega_\ell^{\nu(\tau)}
P_{\tau,f}(\omega_\ell^{-Mr}c)= Q_{\tau,f}(c,\omega_\ell).
\end{equation*}
For the $\GM$-action $\sigma$ on $(\AA_k^1\times \GM)$ given by
$\sigma(\lambda) \cdot (c, \omega_\ell)=(\lambda^{Mr\ell} c,
\lambda^{\ell} \omega_\ell)$, the polynomial $Q_{\tau,f}$ is
homogeneous of degree $\nu(\tau)\ell$.
\end{lem}
\begin{proof}We did already notice that the map $\Phi$ is surjective.
Conversely $\omega$ is determined by $x(\varphi)$ up to a $M$-th
root of unity, and uniquely determined by $x(\varphi)$ and
$y(\varphi)$ for the \emph{gcd} of $M$ and exponents of non zero
terms in $y_\tau(\omega)$ is equal to~1.
\end{proof}

\subsection{}\label{3.3}On the constructible set $\cX_{\tau,\ell}$, via the
isomorphism $\Phi$, we have a morphism to $\AA_k^1\times \GM$
induced by the first projection from $(\AA_k^1\times
\GM)\times\AA_k^{N(\tau,\ell)}$. The constructible set
$\cX_{\tau,\ell}$ defines a class in $\cM_{\AA_k^1\times
\GM}^{\GM}$ we denote by $[\cX_{\tau,\ell}]$. On the other hand,
the function $\ac(f)$ induces a $\GM$-equivariant morphism from
$\cX_{\tau,\ell}$ to $\GM$, hence defines a class in
$\cM_{\GM}^{\GM}$ we denote by $[\cX_{\tau,\ell}(f)]$. By Lemma
\ref{isoxrl}, the morphism $\cX_{\tau,\ell} \to \GM$ is equal to
the composition of the morphism $\cX_{\tau,\ell} \to \AA_k^1\times
\GM$ with $Q_{\tau,f}$.

\subsection{}If $v$ is a rupture vertex of height $r$, there is
only one contact ending in $v$ that we denote by $\tau_v$. We set
$Q_{v,f} := Q_{\tau_v,f}$.

We are now in position  to restate Guibert's theorem
\cite{guibert} in the following form:

\begin{theorem}[Guibert]With the above notation, the following holds:
\begin{equation*}
i_p^* \cS_{f}= [\varpi_{m_p} :\GM\longrightarrow \GM]-\sum_v
\Psi_{Q_{v,f}} ( [\mathrm{Id}:\AA_k^1\times \GM \longrightarrow
\AA_k^1\times \GM])
\end{equation*}
where the second sum runs over the rupture vertices of $\tau(f)$
above $p$.
\end{theorem}

\begin{proof}
Note that for a two variable quasihomogeneous polynomial $Q$
\begin{multline}\label{}
 \Psi_{Q}([\mathrm{Id} :
\AA_k^1\times \GM \longrightarrow  \AA_k^1\times \GM])
 =\\-[Q:(\AA_k^1\times \GM)\setminus
Q^{-1}(0) \longrightarrow \GM]+ \cS_Q([\AA_k^1\times \GM])
\end{multline}
where $\cS_Q$ is defined as in \cite{ivc}, \cite{glm}. We denote
by $\pi_E$ the morphism $(a,b)\longmapsto a^E$ from $\GM\times
\GM$ to $\GM$. When the zeroes of $Q$ are a disjoint union of one
dimensional $\GM$-orbits, $\cS_Q(\AA_k^1\times \GM)$ decomposes
into a sum
 \[
\cS_Q(\AA_k^1\times \GM)=-\sum_i [\pi_{E_i}:\GM\times
\GM\longrightarrow \GM]
 \]
where $E_i$ is the multiplicity of $Q$ along the $i$-th component
of $Q^{-1}(0)$. If $v$ is a rupture vertex  we consider the
following zeta function:
 \[
Z_f^{v}(T):=\sum_{\ell\geq 1}\sum_\tau [\cX_{\tau,\ell}(f)]
\,\LL^{-2\nu(\tau)\ell} \, T^{\nu(\tau)\ell}
 \]
where the second sum is extended to the set of contacts $\tau$
which contain $\tau_v$ and do not contain or end in any successor
of $v$ in the set of rupture vertices. From \cite{guibert} (3.3)
and (5.2), we deduce that $Z_f^{v}(T)$ has a limit $-\cS_f^{v}$ in
the Grothendieck ring $\cM_{\GM}^{\GM}$ when $T$ goes to infinity,
which is given by the formula
\[
\cS_f^v=-\Psi_{Q_{v,f}}([\mathrm{Id} : \AA_k^1\times \GM
\longrightarrow  \AA_k^1\times \GM]).
\]
 We consider the series
\[
Z_f^{p} (T):=\sum_{\ell\geq 1}\sum_\tau [\cX_{\tau,\ell}(f)]
\,\LL^{-2\nu(\tau)\ell} \, T^{\nu(\tau)\ell}
\]
where the second sum is extended to the set of contacts $\tau$
starting from the root corresponding to $p$, which do not contain
or end in any successor of $v$ in the set of rupture vertices.
Again by \cite{guibert},  loc. cit., $Z_f^{p} (T)$ has limit
$-[\varpi_{m_p}
:\GM\longrightarrow \GM]$ when $T$ goes to infinity.
The restriction $i_p^* \cS_{f}$ is the limit as $T \mapsto \infty$
of  $-i_p^* Z_{f}(T)$ which decomposes into
 \[
-i_p^* Z_{f}(T)= - Z_f^{p}(T)-\sum_v Z_f^{v}(T)
\]
where the sum extends to all the rupture vertices of $\tau(f)$
above $p$. The result follows.
\end{proof}

\section{Composition with a morphism}
\subsection{} For $1\leq j\leq 2$, let $g_j:X_j \longrightarrow
\AA_k^1$ be a function on a smooth $k$-variety $X_j$. By
composition with the projection, $g_j$ becomes a function on the
product $X=X_1\times X_2$. We write $d_j$ for the dimension of
$X_j$, $j$ from 1 to 2, and $d$ for $d_1+d_2$. Define $\mathbf{g}$
as the map $g_1\times g_2$ on $X$ and $G$ as the product
$G=g_1g_2$. For any subvariety $Z$ of the set $X_0(G)$ containing
$X_0(\mathbf{g}):=g_1^{-1}(0)\cap g_2^{-1}(0)$, we denote by $i$
the closed immersion of $X_0(\mathbf{g})$ in $Z$.

As in section \ref{section2}, we denote by $f$ a two variable
polynomial, we assume that $f(0,y)$ is a  nonzero polynomial, we
denote by $p$ the origin and will denote by $m_p$  the order of
$0$ as a root of $f(0,y)$. We consider the augmented set of
rupture vertices of the tree $\tau(f,p)$, namely the set of
rupture vertices together with the vertex of minimal non zero
height on $\tau(f,p)$. Denote that vertex by $v_0$ and consider
its associated polynomial $Q_{v_0}$. We denote by $\cS'_{g_2}$ the
element in $\cM_{X_0(g_2)\times \AA_k^1}^{\GM}$ which is the
``disjoint sum" of $\cS_{g_2}$ in $\cM_{X_0(g_2)\times \GM}^{\GM}$
and $X_0(g_2)$ in $\cM_{X_0(g_2)}$. We set $A_{v_0}:=
\cS'_{g_2}\boxtimes \cS_{g_1}$, considered as an element in
$\cM_{X_0(\mathbf{g})\times (\AA_k^1\times \GM)}^{\GM}$. For any
rupture vertex $v$ of the tree $\tau(f,p)$, we denote by $a(v)$
the predecessor of $v$ in the augmented set of rupture vertices
and we define by induction a virtual variety $A_v$ in
$\cM_{X_0(\mathbf{g})\times \AA_k^1\times \GM}^{\GM}$. We assume
we are given a virtual variety $A_{a(v)}$  in
$\cM_{X_0(\mathbf{g})\times \AA_k^1\times \GM}^{\GM}$ whose
restriction over $X_0(\mathbf{g})\times \GM\times\GM$ is
diagonally monomial in the sense of \cite{ivc} (2.3), or more
precisely whose $\GM$-action is diagonally induced from a
diagonally monomial $\GM^2$-action in the sense of \cite{ivc}. To
any successor of $a(v)$ corresponds a factor of the polynomial
$Q_{a(v)}$. Denote by $Q_{a(v)}^v$ the factor associated to $v$.
Notice that $(Q_{a(v)}^v)^{-1}(0)$ is a smooth subvariety in $\GM
\times \GM$ equivariant under a diagonal $\GM$-action and that the
second projection $\pr_2$ of the product $\AA_k^1\times \GM$
induces an homogeneous fibration from $(Q_{a(v)}^v)^{-1}(0)$ to
$\GM$. We denote by $B_v$ the restriction of $A_{a(v)}$ above
$(Q_{a(v)}^v)^{-1}(0)$. The external product of the identity of
the affine line $\AA^1_k$ by the induced map $\pr_2:
B_v\longrightarrow \GM$ defines an element $A_v$ in
$\cM_{X_0(\mathbf{g})\times (\AA_k^1\times \GM)}^{\GM}$ which is
diagonally monomial when restricted to $X_0(\mathbf{g})\times
\GM\times\GM$.

\begin{theorem}\label{main}With the previous notations
and hypotheses, the  following formula holds
\begin{equation}\label{}
i^* \cS_{f\circ \mathbf{g}} = \cS_{{(g_2)}^{m_p}}(X_0(g_1)) -
\sum_v \Psi_{Q_v} (A_v),
\end{equation}
where the sum is runs over  the augmented set of rupture vertices
of the tree 
 $\tau(f, p)$. 
\end{theorem}
\begin{proof}We first reduce to the case where $\tau(f,p)$ has only a
finite number of vertices.
\begin{lem}There exists a rational number $\gamma$ such that, for any $r$
greater than~$\gamma$,
\begin{equation}\label{}
i^*\cS_{f\circ \mathbf{g}} = i^*\cS_{f_{\tau_r}\circ \mathbf{g}}.
\end{equation}
\end{lem}
\begin{proof}
Consider a rupture vertex $v$. The quotient $\ord_t(f\circ
\mathbf{g}(\varphi))/\ord_t (g_1(\varphi))$ is an affine function
of $r$ whenever $\tau$ contains $\tau_v$ and does not contain or
end in any rupture vertex greater than $v$. Hence the quotient
$\ord_t(f\circ \mathbf{g}(\varphi))/\ord_t (g_1(\varphi))$ is a
function on the tree $\tau(f)$, the restriction of which on each
semi-open branch joining two consecutive rupture vertices (resp.
on each infinite branch above a rupture vertex) is an increasing
affine function of the height.

We consider the following set of arcs
\begin{equation*}
\cX_{n_1,n_2}({x\circ \mathbf{g}, f\circ \mathbf{g}}) := \Bigl
\{\varphi \in \cL_{n_1+n_2} (X) \Bigm \vert \ord_tx\circ
\mathbf{g}=n_1,\; \ord_tf\circ \mathbf{g}=n_2\Bigr\}.
\end{equation*}
For 
$\gamma$
large enough, the zeta function
\begin{equation*}
Z^\gamma_{x\circ \mathbf{g}, f\circ \mathbf{g}}(T)=\sum_{n_2\geq
\gamma n_1}[\cX_{n_1,n_2}({x\circ \mathbf{g}, f\circ
\mathbf{g}})]\LL^{-(n_1,n_2)d}T^{n_2}
\end{equation*}
goes to zero as $T$ goes to infinity, cf. \cite{ivc}. The lemma
follows.
\end{proof}

To prove the theorem, it is enough to consider the case when the
tree $\tau(f,p)$ has a finite number of vertices. The proof goes
by induction on 
the number of vertices
of 
the tree $\tau(f,p)$. Certainly the result holds if there is no
vertex. Assume first the tree has only one vertex $v_0$. The
formula is then a particular case of the main formula in
\cite{glm}. Assume now we have at least two vertices. Choose a
maximal vertex $v$ for the height function on $\tau(f,p)$ and
consider the subtree $\tau^-$ obtained from $\tau(f,p)$ by
deleting the vertex $v$. Denote by $a(v)$ the predecessor of $v$
on $\tau(f,p)$ and by $f^{-}$ the polynomial associated to
$\tau^-$.

Consider an arc $\varphi$ in $\AA_k^2$ with origin $p$. Then one
of the  following two statements holds:
\begin{itemize}
    \item[$\bullet$] The contact of $\varphi$
    with $f$ does not contain $\tau_v$. Then
    $\ord_t(f(\varphi))=\ord_t(f^-(\varphi))$ and
    $\ac(f(\varphi))=\ac(f^-(\varphi))$.
    \item[$\bullet$] The contact of $\varphi$ with $f$  contains $\tau_v$.
\end{itemize}
According to these two different cases, we can split the zeta
function $Z_{f\circ \mathbf{g}}$ in two pieces, namely
\begin{equation}\label{}
Z_{f\circ \mathbf{g}}=Z_{<v}+Z_{\geq v},
\end{equation}
Similarly, the zeta function $Z_{f^-\circ \mathbf{g}}$ decomposes
in
\begin{equation}\label{}
Z_{f^-\circ \mathbf{g}}=Z_{<v}^-+ Z_{\geq v}^-.
\end{equation}
We noticed that $Z_{<v}^-=Z_{<v}$, hence we get
\begin{equation}\label{}
Z_{f\circ \mathbf{g}}- Z_{f^-\circ \mathbf{g}} =Z_{\geq v}-Z_{\geq
v}^-.
\end{equation}
In section \ref{3.3}, for any contact $\tau$ and integer $\ell$,
we have considered a set $\cX_{\tau,\ell}$ associated to a
polynomial $f$. Similarly we have a set $\cX_{\tau,\ell}^-$
associated to $f^-$. These two sets  map to $\AA_k^1\times \GM$.
Consider now the inverse image by $\mathbf{g}$ of
$\cX_{\tau,\ell}$ (resp. $\cX_{\tau,\ell}^-$) and denote it by
$\cX_{\tau,\ell}(\mathbf{g})$ (resp.
$\cX_{\tau,\ell}^-(\mathbf{g})$). We assume, by induction, that
the motivic nearby cycles of $f^-$ have the given form and that
for any contact $\tau$ greater than $\tau_{a(v)}$ the set
$\cX_{\tau,\ell}^-(\mathbf{g})$ is a piecewise affine bundle on
$X_0 (\mathbf{g}) \times \GM \times B_{a(v)}$.

An arc $\mathbf{g}\circ \varphi$ in $\cL_{\nu(\tau)\ell}(\AA^2_k)$
having contact $\tau$ with $f^-$ has contact $\tau_v$ with $f$ if
and only if $\tau=\tau_v$ and $Q_{a(v)}^v$ does not vanish on
$\varphi$ or if $\tau$ contains strictly $\tau_v$. In that case
$\varphi$ maps to $\{0\}\times B_{a(v)}$. Hence the set
$\cX_{\tau_v,\ell}(\mathbf{g})$ is a disjoint union of piecewise
affine bundles on $X_0 (\mathbf{g}) \times ((\AA^1_k \times
B_{a(v)}) \setminus B_v)$ and the function $\ac(f\circ
\mathbf{g})$ is given by the composition of the canonical map with
$Q_v$.

An arc $\mathbf{g}\circ \varphi$ in $\cL_{\nu(\tau)\ell}(\AA^2_k)$
has contact greater than $\tau_v$ with $f$ if and only if it has
contact $\tau_v$ with $f^-$ and $Q_{a(v)}^v$ vanish on $\varphi$.
Hence, for any contact $\tau$ greater than $\tau_v$, the set
$\cX_{\tau,\ell}(\mathbf{g})$ is a disjoint union of piecewise
affine bundles on $X_0 (\mathbf{g}) \times \GM\times B_v$ and the
function $\ac(f\circ \mathbf{g})$ is given by the composition of
the canonical map with the projection $X_0 (\mathbf{g}) \times
\GM\times B_v\longrightarrow \GM$.

We can compute the difference $Z_{\geq v}-Z^-_{\geq v}$ and check
that it has limit $\Psi_{Q_v}(A_v)$ as $T$ goes to infinity.

It is a consequence of  the following lemma, which follows from
direct computation, that only extended rupture vertices have a non
zero contribution.
\end{proof}

\begin{lem}
Consider a vertex $v$ of $\tau(f,p)$ and assume it is not
an extended
rupture vertex. Then the polynomial $Q_v$ is of the form:
$Q_v(c,\omega)=((c-\alpha\omega^R)\omega^N)^E$ where $R$, $N$ and
$E$ are integers and $\alpha$ a non-zero constant. It defines a
map from $\AA^1_k\times \GM$ to $\AA^1_k$ the zero set of which is
isomorphic to $\GM$. Then:
\begin{itemize}
    \item[$\bullet$]  $\Psi_{Q_v}(A_v)=0$ in
$\cM_{X_0(\mathbf{g})\times\GM}^{\GM}$.
    \item[$\bullet$] For
the unique successor $s(v)$ of $v$, the equality $A_{s(v)}=A_v$
holds  in $\cM_{X_0(\mathbf{g})\times\AA^1_k\times\GM}^{\GM}$.
\end{itemize}
\end{lem}

\bibliographystyle{amsplain}

\end{document}